\documentclass{amsart}

\usepackage{amsmath}
\usepackage{amssymb}
\usepackage{graphicx}
\usepackage{enumerate}

\newtheorem{theorem}{Theorem}[section]
\newtheorem{lemma}[theorem]{Lemma}
\newtheorem{proposition}[theorem]{Proposition}
\newtheorem{corollary}[theorem]{Corollary}

\newtheorem{definition}[theorem]{Definition}
\newtheorem{remark}[theorem]{Remark}

\begin{document}

\title{Smoothings of singularities and symplectic surgery}

\author{Heesang Park}

\address{School of Mathematics, Korea Institute for Advanced Study, Seoul
  130-722, Korea}

\email{hspark@kias.re.kr}

\author{Andr\'{a}s I. Stipsicz}

\address{R\'enyi Institute of Mathematics, Budapest, Hungary}

\email{stipsicz.andras@math-inst.hu}

\maketitle

\begin{abstract}
Suppose that $C=(C_1, \ldots , C_m)$ is a configuration of
2-dimensional symplectic submanifolds in a symplectic 4-manifold $(X,
\omega)$ with connected, negative definite intersection graph $\Gamma
_C$. We show that by replacing an appropriate neighborhood of $\cup
C_i$ with a smoothing $W_S$ of a normal surface singularity $(S, 0)$
with resolution graph $\Gamma _C$, the resulting 4-manifold admits a
symplectic structure. This operation generalizes the rational
blow-down operation of Fintushel-Stern for other configurations, and
therefore extends Symington's result about symplectic rational
blow-downs.
\end{abstract}

\section{Introduction}
\label{sec:intro}
Suppose that $X$ is a closed, oriented 4-manifold.  Recall that in the
rational blow-down procedure (introduced by Fintushel and Stern
\cite{FS} and extended by J. Park \cite{jpark}) the tubular
neighbourhood of a collection of embedded spheres $S=(S_1, \ldots ,
S_k)$ in $X$ is replaced by a specific compact 4-manifold $W_S$ with
boundary, providing the closed 4-manifold $X_S$.  The spheres
intersect each other according to a linear graph $\Gamma _S$, and
their self-intersections are determined by the continued fraction
coefficients of the ratio $-\frac{p^2}{pq-1}$ for some $p>q>0$
relatively prime integers. The 4-manifold $W_S$ in the construction
can be interpreted as a particular smoothing of the singularity which
has the plumbing graph $\Gamma _S$ as its resolution graph.  (In fact,
in the rational blow-down construction $W_S$ admits the further
property that it is a rational homology disk, that is, $H_*(W_S; {\mathbb {Q}}
)\cong H_* (D^4; {\mathbb {Q}} )$.)

The success of the rational blow-down construction stems from the fact
that it produces 4-manifolds with interesting differential topology.
The curiosity of the resuting smooth structure can be measured by the
Seiberg-Witten invariants of $X_S$. In fact, for the rational
blow-down construction there is a simple relation between the
Seiberg-Witten invariants of the 4-manifold $X$ and the resulting
4-manifold $X_S$ \cite{FS}.  In specific cases the nonvanishing of the
Seiberg-Witten invariant of $X_S$ can be explained using symplectic
topology: according to a result of Symington \cite{Sym1, Sym2}, if
$(X, \omega )$ is a symplectic 4-manifold and the spheres in the
configuration $S$ are symplectic submanifolds (intersecting
$\omega$-orthogonally), then $X_S$ admits a symplectic structure
(hence by Taubes' theorem \cite{Taubes} it has nontrivial
Seiberg-Witten invariants).  This symplectic feature of the
construction has been extended to further configurations of symplectic
surfaces in symplectic 4-manifolds and further smoothings of
singularities in \cite{AO, GM, GS1, G-S}. The general case, however, stayed
open and was formulated as Conjecture~1.4 in \cite{G-S}. The aim of
the present paper is to prove this conjecture. Informally, the result
says that if we have a configuration of symplectic surfaces in a
symplectic 4-manifold which intersect each other according to a
negative definite matrix, we collapse them to a point, and deform the
resulting complex singularity, then the deformation 'globalizes in the
symplectic category'.

To formulate the theorem precisely, suppose that $(X, \omega )$ is a
closed symplectic 4-manifold and $C=(C_1, \ldots , C_m)$ is a
collection of smooth, closed 2-dimensional submanifolds which satisfy
the following properties:
\begin{itemize}
\item each $C_i$ is a symplectic submanifold and $C=\cup C_i$ is connected,
\item $C_i$ intersects each other $C_j$ $\omega$-orthogonally in at most one
  point, and
\item the intersection matrix $I=(C_i \cdot C_j)$ (with the
self-intersections in the diagonal) is negative definite.
\end{itemize}
Suppose that $\Gamma _C$ is the (connected) plumbing graph
corresponding to the curve configuration $C$. By our assumption it is
a negative definite plumbing graph, where the vertex $v$ corresponding
to the surface $C_v$ is decorated by the self-intersection
$e_v=C_v\cdot C_v<0$ and by the genus $g_v=g(C_v)\geq 0$. According to
a fundamental result of Grauert \cite{grau}, for any such graph there
is a normal surface singularity $(S,0)$ with resolution dual graph
equal to $\Gamma _C$.  (The analytic type of the singularity is not
necessarily unique, although its topology is determined by the graph
$\Gamma _C$.) Suppose that $W_C$ is a smoothing (or Milnor fiber) of
the chosen singularity $(S,0)$.

The main result of the present paper then reads as follows.
\begin{theorem}\label{thm:main}
  With the notations above, let $\nu C$ denote a tubular neighbourhood of the
  union $\cup C_i$. Then, under the assumptions listed above, there is an
  orientation-reversing diffeomorphism $\phi \colon \partial (X-{\rm {int
    }}\nu C)\to \partial W_C$ such that the glued-up manifold
\[
X_C=(X-{\rm {int}}\nu C)\cup _{\phi} W_C
\]
admits a symplectic structure which is equal to the restriction of
$\omega$ over $X-\nu C$.
\end{theorem}

The main idea of the proof of the above result is the following: by
\cite{G-S} the configuration $C=\cup _i C_i$ admits an $\omega$-convex
neighbourhood $U_C$, with boundary $\partial U_C$ supporting a
compatible contact structure $\xi _C$. The smoothing $W_C$, on the
other hand, admits the structure of a Stein domain, inducing the
so-called \emph{Milnor fillable} contact structure $\xi _M$ on
$\partial W_C$.  By our assumption $\partial U_C$ and $\partial W_C$
are orientation preserving diffeomorphic 3-manifolds. The main tool in
verifying Theorem~\ref{thm:main} is the result showing that $\xi _C$
and $\xi _M$ are, in fact, contactomorphic. Therefore taking a
contactomorphism $\psi \colon (\partial U_C, \xi _C)\to (\partial W_C,
\xi _M)$, the gluing of symplectic 4-manifolds along contact
hypersurfaces (as it is explained in \cite{Etn}, cf. also \cite{OS})
concludes the argument.  In turn, the fact that the two contact
structures $\xi _C$ and $\xi _M$ are contactomorphic will be proved by
relating two compatible open book decompositions. The existence of
this contactomrphism was verified in \cite{AO, G-S} for specific
families of configurations of $C$; in this paper we extend the result
of \cite{G-S} by constructing an appropriate horizontal open book
decomposition on $\partial U_C$ compatible with $\xi _C$
(cd. Theorem~\ref{theorem-HOB}.)

The paper is organized as follows. In Section~\ref{sec:hob} we quickly
recall some facts about horizontal open book decompositions, and in
Section~\ref{sec:comp} we give the proof of the main result of the
paper. In Section~\ref{sec:example} we show an example where the
resolution graph involves curves of high genus.

{\bf Acknowledgements:} HP would like to thank Andras N\'{e}methi for valuable
discussions. The main part of the work in this paper was done during HP's
visit (supported by W-PR program of KIAS) to the Alfr\'{e}d R\'{e}nyi
Institute of Mathematics, whose hospitality and support are gratefully
acknowledged. HP was supported by Basic Science Research Program through the
National Research Foundation of Korea (NRF-2011-0012111). AS was supported by
the \emph{Lend\"ulet} project of the Hungarian Academy of Sciences and by ERC
Grant LTDBud.

\section{Horizontal open book decompositions}
\label{sec:hob}

By the Giroux correspondance \cite{Gi}, open book decompositions
play a central role in contact topology. For completeness, in this section
we recall some facts and constructions regarding specific open book
decompositions on plumbed 3-manifolds. We start with a general definition.

\begin{definition}
  Suppose that $Y$ is a given closed, oriented 3-manifold.  The pair
  $(B, \varphi )$ is an {\bf open book decomposition} on $Y$ if
  $B\subset Y$ is an oriented 1-dimensional submanifold and $\varphi
  \colon Y-B \to S^1$ is a locally trivial fibration with the property
  that a fiber $\varphi ^{-1}(t)$ is the interior of a Seifert surface
  of $B$. The submanifold $B$ is called the {\bf binding} of the open
  book, while the closure of a fiber $\varphi ^{-1}(t)$ is called a
  {\bf page}.  Two open book decompositions $(B,\varphi )$ and $(B',
  \varphi ')$ of the diffeomorphic 3-manifolds $Y$ and $Y'$ are {\bf
    equivalent} if there is an orientation preserving diffeomorphism
  $g \colon Y \to Y'$ with the properties that $g (B)=B'$ (as oriented
  1-manifolds) and $\varphi =\varphi '\circ g$.
\end{definition}

According to the Giroux correspondence \cite{Gi}, an open book decomposition
uniquely determines an isotopy class of compatible contact structures. Recall
that the contact form $\alpha$ is \emph{compatible} with the open book
decompositions $(B, \varphi )$ if $B$ is tangent to the Reeb flow $R_{\alpha
}$ defined by $\alpha$, while the interiors of the pages are transverse to the
Reeb flow.

By a classical result of Stallings \cite{Stall}, in a rational homology
3-sphere an open book decomposition is determined by its binding. For
manifolds with $b_1>0$ this principle no longer holds. By considering specific
classes of 3-manifolds and open book decompositions, however, a similar
statment can be proved.  For the statement we need a little preparation. (For
related notions, see also \cite{EO}.) Suppose that $Y$ is a graph manifold,
that is, it is given by the plumbing construction along a weighted graph
$\Gamma$.  This means that we consider circle bundles over the surfaces
corresponding to the vertices (with Euler numbers specified by the framings of
the graph) and plumb these pieces together.
\begin{definition}
  An open book decomposition $(B, \varphi )$ on a graph manifold $Y$ is {\bf
    horizontal} if the binding $B$ is the union of fibers of the individual
  fibrations, the pages are transverse to these fibrations and the orientation
  induced by the pages on the binding coincides with the orientation given by
  the fibration.
\end{definition}
Let ${\bf {n}}=(n_v)$ denote the vector of nonnegative numbers specified by
the binding components at each vertex $v$ of the plumbing graph $\Gamma$.  Now
the version of Stallings' result (due to Caubel--N\'emethi--Popescu-Pampu) is
the following:
\begin{theorem}[\cite{C-N-P}]\label{thm:equiv}
Suppose that $(B,\varphi )$ and $(B',\varphi ')$ are two horizontal open book
decompositions of the plumbing 3-manifold $Y=Y_{\Gamma}$.
If ${\bf {n}}={\bf {n'}}$ and for each vertex $v$ we have $n_v=n'_v>0$ then
the two open book decompositions are equivalent, and hence the
compatible contact structures are contactomorphic. \qed
\end{theorem}

By another result of Caubel--N\'emethi--Popescu-Pampu,
horizontal open book decompositions
compatible with the Milnor fillable contact structure $\xi _M$
are easy to construct:
\begin{proposition}[Theorem 4.1 of \cite{C-N-P}]\label{pro-CNP}
  Let $p:({\tilde {S}}, E) \to (S, 0)$ be a good resolution of a
  normal surface singularity $(S, 0)$, where $E$ is a normal crossing
  divisor in ${\tilde {S}}$ having smooth components $E_1,\ldots ,E_m$
  with $E=\sum _i E_i$. Assume that the nonzero effective divisor
  $D=\sum d_i E_i$ ($d_i \in \mathbb{N}$) satisfies
\begin{equation*}
(D+E+K_{\Sigma})\cdot E_i +2 \leq0  \text{ for any $i=1,\ldots ,m$.}
\end{equation*}
Then there exists a holomorphic function on $(S, 0)$ with an isolated
singularity at $0$ such that $div(f\circ p)$ is a normal crossing divisor on
${\tilde {S}}$ and the exceptional part of $div(f\circ p)$ is $D$. Moreover,
for each $i$, the number of intersection points $n_i=div(f\circ p)_s \cdot
E_i$ is the strictly positive, where $div(f\circ p)_s$ is strict transform
part of $div(f\circ p)$. \qed
\end{proposition}

By \cite[Remark~4.2]{C-N-P}, for any good resolution of the normal
surface singularity $(S, 0)$ there is a effective $D$ which satisfies
the condition of Proposition~\ref{pro-CNP}. Consider now the open book
decomposition determined by a function $f$ provided by
Proposition~\ref{pro-CNP}: let $B=f^{-1}(0)\cap \partial W_C$ and
$\varphi =\frac{f}{|f|}$. As it was explained in \cite{C-N-P}, the
resulting open book decomposition is horizontal, compatible with $\xi
_M$, and with the notation $n_v=-D\cdot E_v$ each $n_v$ is strictly
positive.

Therefore there are horizontal open book decompositions compatible
with the Milnor fillable contact structure $\xi _M$, and indeed, we
can find such open books with the extra condition that $n_v>0$ for all
$v$. A useful simple observation shows that if ${\bf {n}}=(n_v)$
appears as such a vector, then so does $k\cdot {\bf {n}}$ for any
positive integer $k$:

 \begin{lemma}\label{lemma-multiple}
   Let $D$ be an effective divisor which satisfies the conditions of
   Proposition~\ref{pro-CNP}. Then any positive integer multiple
   $k\cdot D$ of $D$ also satisfies those conditions.
 \end{lemma}
 \begin{proof}
   Let $D=\sum d_i E_i$ be an effective divisor satisfying $(D+E+K)E_i + 2
   \le 0$ for all $i=1,\dots ,m$. Let $k$ be a positive integer. Then
   $k((D+E+K)E_i+2) \le 0$ for all $i$. Since $0 \leq \sum_{j \neq i}{E_jE_i}
   + 2g(E_i) = (E-E_i)\cdot E_i + (E_i+K)\cdot E_i+2=
(E+K)E_i + 2 \le k((E+K)E_i+2)$ for all $i$, we have
\begin{equation*}
(k\cdot D+E+K)E_i + 2 \le k((D+E+K)E_i +2) \le 0
\end{equation*}
for all $i$.
 \end{proof}

\section{Horizontal open books  for $\xi _C$}
\label{sec:comp}

Now we turn our attention to constructing horizontal open book decompositions
compatible with the contact structure $\xi _C$.  First of all, following
\cite[Section~4]{G-S} we extend the notion of an open book decomposition for
manifolds with boundary as follows: if $M$ is a compact 3-manifold with
nonempty boundary $\partial M$, then $(B, \varphi )$ is an open book
decomposition if $B\subset M-\partial M$ is an oriented link and $\varphi
\colon M-B \to S^1$ is a map which behaves near $B$ as a usual open book does
and restricts to $\partial M$ as a fibration $\partial M \to S^1$.

Suppose that $Y$ is a plumibing 3-manifold along the plumbing graph
$\Gamma$.  Let $v$ be a fixed vertex of the graph $\Gamma$ and suppose
that $\{ v_1, \ldots , v_k\}$ are the further vertices connected to
$v$ in the graph.  Recall that $e_v$ denotes the framing fixed for
$v$. Let $N_v$, $N_{v_j}$(for $j=1,\ldots ,k$) and $n_v$ be positive
integers satisfying
\begin{equation}\label{eq:condition}
N_v e_v+\sum N_{v_j} =- n_v.
\end{equation}
In short, if $I$ denotes the intersection matrix of $\Gamma$ (with the
$e_v$'s in the diagonal), then ${\bf {N}}\cdot I =-{\bf {n}}$, where
${\bf {N}}=(N_v)$ and ${\bf {n}}=(n_v)$.

Let $D^2$ be a 2-disk containing disjoint small disks $D_1, \ldots
,D_k$ and let $A=D^2-\cup _{i=1}^k {\rm {int }} D_i$. Consider $M=A
\times S^1$ with coordinates $\beta \in S^1$ and $\gamma_j \in
\partial D_j$ and $\alpha \in \partial D^2$ ($\alpha$ and $\gamma_j$ with
orientation as boundary of $D^2$ and $D_j$ respectively). Now an adaptation of
\cite[Lemma~4.1]{G-S} gives the following result.
\begin{lemma}\label{lemma-HOB}
  There exists an open book decomposition $(B, \varphi )$ on $M=A\times S^1$
  such that the following conditions hold:
\begin{enumerate}
\item[(1)]$\varphi|_{\partial D^2 \times S^1} =-e_vN_v\alpha+ N_v\beta$.
\item[(2)]$\varphi|_{\partial D_j\times S^1} =N_{v_j}\gamma_j + N_v\beta$.
\item[(3)]The pages $\varphi ^{-1}(\theta)$ are transverse to $\partial_\beta$.
\item[(4)]The binding $B$ is tangent to $\partial_\beta$.
\item[(5)]$B$ has $n_v$ components $B_1,\ldots ,B_{n_v}$.
\item[(6)]When the pages are oriented so that $\partial _\beta$ is
  positively transverse, then $B_1,\ldots ,B_{n_v}$ are oriented in the
  positive $\partial _\beta$ direction.
\end{enumerate}
\end{lemma}
 \begin{proof}
   Let $p_1, \ldots , p_{n_v}$ be fixed points in $A$. We may assume that the
   centers of the disks $D_1,\ldots ,D_k$ and the fixed points $p_1,\ldots
   ,p_{n_v}$ lie on a line segment and that each $D_j$ and $p_i$ are contained
   in the interior of another disk $D'_j$ for $1 \leq j \leq k+n_v$ such that
   the disks $D'_j$ and $D'_{j+1}$ are tangent to each other at one point and
   the center of $D'_j$ is equal either to the center of $D_j$ or to
   $p_{j-k}$; see Figure~\ref{fig1}.
\begin{figure}
\includegraphics{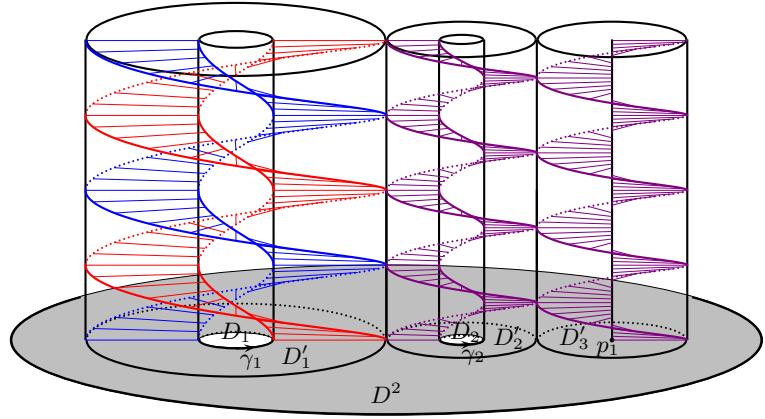}
 \caption{The diagram shows the special case when the vertex $v$ has
   two neighbours $v_1, v_2$, furthermore $N_v =4, N_{v_1}=2,
   N_{v_2}=1$, and the given value of $ n_v$ is equal to $1$.}
 \label{fig1}
 \end{figure}
The desired open book decomposition on $M=A\times S^1$ will be built
from pieces.

First we consider an index $j$ between $1$ and $k$ and describe the open book
decomposition on $(D_j'-D_j)\times S^1$ over the annulus $D_j'-D_j$. Each such
index $j$ corresponds to an other vertex $u$ of the plumbing graph $\Gamma$
with the property that $v$ and $u$ are joined by an edge. Now consider the
curve $(N_v, -N_u)$ on $\partial ((D_j'-D_j)\times S^1)$.  Notice that the
boundary $\partial ((D_j'-D_j)\times S^1)$ has two components, an inner and an
outer one. In addition, the positive integers $N_v$ and $N_u$ are not
necessarily relatively prime, therefore the resulting curve on one of the
boundary components is not necessarily connected.  Foliate the boundaries by
these curves, and extend these foliations to a foliation of $(D_j'-D_j)\times
S^1$ by (possibly disjoint unions of) annuli. cf. the left hand portion of
Figure~\ref{fig1}.

Consider now an index $j$ between $k+1$ and $k+n_v$. The disk $D_j'$ then
contains $p_{j-k}$.  The open book decomposition on $D_j'\times S^1$ will have
binding equal to $\{ p_{j-k}\} \times S^1$ and the pages provide a foliation
of $\partial (D_j'\times S^1)$ by curves of type $(N_v, -1)$. (These
conditions uniquely determine the open book decomposition.) For an
illustrative example, see the right hand portion of Figure~\ref{fig1}.

The union of the above pieces now provide an open book on
$(\cup D_j')\times S^1$. After trivial smoothings over the
tangencies of the consecutive disks, this construction then
extends to an open book decomposition on $A\times S^1$.

The proofs of the properties are routine exercises inside the individual
pieces. There is one property which needs to be commented: we need to compute
the slope of the fibration given by the pages on the outer boundary component
$\partial D^2 \times S^1$. By construction, the curves we get by intersecting
the boundary with the pages are of type $(N_v, x)$, where $x$ is the sum of
the corresponding coordiantes on the individual pieces, cf. Figure~\ref{fig2}.
\begin{figure}
\includegraphics{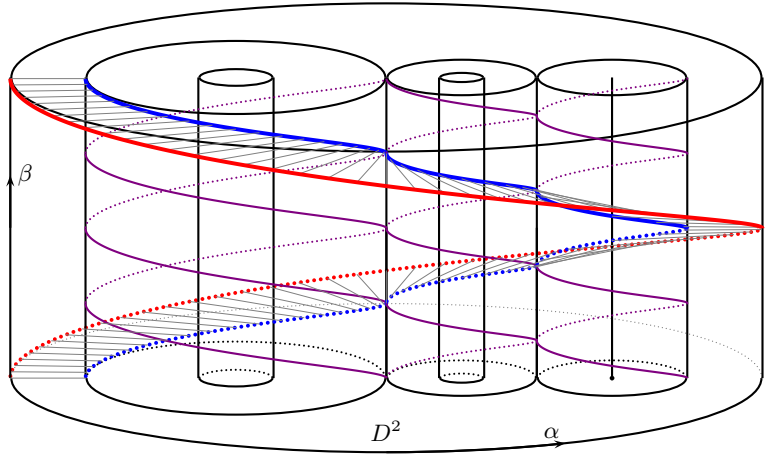}
\caption{Computation of the slope of the curves on boundary of the union
  $(\cup D_j')\times S^1$. The diagram depicts part of the page we get
by gluing the pages of the individual open book decompositions together.}
 \label{fig2}
 \end{figure}
By our choices, we get that $x=-n_v -\sum N_{u}
$ where the summation goes for those vertices $u$ which are connected
in the graph $\Gamma$ to $v$.  By our choice of the vector ${\bf {N}}$
(given in Equation~\eqref{eq:condition}) we get that $x=N_v e_v$, verifying
the first claimed property as well.
\end{proof}

According to a standard result (cf.  \cite[Corollary~3.4]{G-S}, for example),
for a negative definite symmetric matrix $I$ with nonnegative off-diagonals,
and for any $\mathbf{n} \in (\mathbb{R}^+)^m$ the vector $-\mathbf{n}\cdot
I^{-1}$ is in $(\mathbb{R}^+)^m$.

\begin{theorem}\label{theorem-HOB}
  Suppose that $Y=Y_{\Gamma}$ is a given plumbing manifold, where
  $\Gamma$ is a negative definite plumbing graph on $m$
  vertices. Suppose that ${\bf {n}}=(n_1, \ldots , n_m)$ is a given
  vector in ${\mathbb {N}} ^m$. Then there is $k\in {\mathbb {N}}$ and
  a horizontal open book decomposition $(B, \varphi )$ on $Y$ which is
  compatible with $\xi _C$ and at each vertex $v$ it has $kn_v$
  binding components.
\end{theorem}
\begin{proof}
  Let $\Gamma $ be the given negative definite plumbing graph with
  vertices corresponding to the surfaces $C_v$ ($v=1, \ldots , m$)
  with self-intersection numbers $C_v ^2=e_v$ and genera
  $g(C_v)=g_v$. The graph defines a 4-manifold (containing the
  surfaces $C_v$) and $Y$ is the boundary of this plumbing 4-manifold.

For the given positive integral vector ${\mathbf {n}}$ consider $
\mathbf{N}\in {\mathbb {Q}}^m $ satisfying the relation given by
Equation~\eqref{eq:condition} ${\mathbf {N}} \cdot I = -\mathbf{n}$,
where (as before) $I$ is a intersection matrix of the plumbing
graph. The observation preceding Theorem~\ref{theorem-HOB} implies
that $N_v$ is nonnegative for all $v$.  Notice that \emph{a priori}
each $N_v$ is a rational number, but after multiplying both sides of
Equation~\eqref{eq:condition} by an appropriate positive integer $k$,
we may assume that the resulting vector (which we still denote by
${\mathbf {N}}$) is integral.

For constructing a horizontal open book decomposition on $Y$ with the
required properties, we decompose the neighborhood of the union
$C=\cup C _v$ into fibered pieces: For the vertex $v$ of the graph
$\Gamma$ we consider the $S^1$-bundle over $C_v$ with Euler number
$e_v$.  Let $\widehat{C}_v$ denote the punctured surface $C_v-
D^2$. Suppose that $\partial \widehat{C} _v \times S^1$ has
coordinates such that $\partial \widehat{C}_v\times \{ 1\}=m_v$, $\{
p\}\times S^1=l_v$, and similarly $\partial D^2\times S^1$ has
coordinates $\partial D^2\times \{ 1\}=\alpha_v$ and $\{ p'\}\times
S^1=\beta _v$. We orient $\alpha_v$ by the boundary orientation of $D^2$
and $m_v$ with the orientation opposite to the boundary orientation of
$\widehat{C}_{v}$. The $S^1$-bundle over $C_v$ with
Euler number $e_v$ is given by gluing $D^2 \times S^1$ and
$\widehat{C}_{v} \times S^1$ with the gluing map given by
\begin{equation}\label{eq:gluing}
\alpha_v + e_v \beta _v \rightarrow m_v \qquad {\rm {and}} \qquad
 \beta_v \rightarrow l_v.
\end{equation}
Assume that each $C_v$ meets the union $\cup _{u\neq v}C_u$ in $k_v$ points.
For each such intersection point (that is, for each index $i$ between $1$ and
$k_v$) choose a small disk $D_{i}\subset D^2$. Let $A_v=D^2-{\rm {int }}
(D_{1}\cup \ldots \cup D_{k_v} ) \subset C_{v}$ be the complement of (the
interiors) of the chosen disks. Near $C_v$ therefore we can decompose the
3-manifold as the union of $A_v \times S^1$ and
$\widehat{C}_{v} \times S^1$.

On $A_v\times S^1$ we take the horizontal open book decomposition
provided by Lemma~\ref{lemma-HOB}.  On $\widehat{C}_{v} \times S^1$ we
consider the horizontal open book decomposition given by the map
$\pi:\widehat{C}_v \times S^1 \rightarrow S^1$ defined as $\pi=N_i
l_i$.

By Property~(1) of Lemma~\ref{lemma-HOB}, when we glue $A_v\times S^1$
and $\widehat{C}_{v} \times S^1$ via the gluing map specified by
Equation~\eqref{eq:gluing}, the open book decompositions also glue
together. Therefore we have a horizontal open book decomposition near
the surface $C_v$.

Let $q$ be an intersection point of $C_v$ and $C_u$ with $q \in D_{i_1}
\subset C_v$ and $q \in D_{i_2} \subset C_u$. When we plumb the two bundles at
$q$, we glue the circle bundles with the map $\gamma _{i_1} \rightarrow
\beta_{i_2}$ and $\beta_{i_1} \rightarrow \gamma _{i_2} $ (where the curves
$\gamma _j$ are as in Lemma~\ref{lemma-HOB}). Thus $N_{i_2}\gamma _{i_1} -
N_{i_1}\beta_{i_1}$ maps to $-(N_{i_1}\gamma _{i_2} - N_{i_2}\beta_{i_2})$. By
Property~(2) of Lemma~\ref{lemma-HOB} the curve $N_{i_2}\gamma _{i_1} -
N_{i_1}\beta_{i_1}$ is part of the boundary of a page of the open book
decomposition on $A_v\times S^1$.  After plumbing, the pages of the open book
decomposition are obtained by gluing pages of the open book decompositions on
$A_v\times S^1$ and $A_u\times S^1$ along $N_{i_2}\gamma _{i_1} -
N_{i_1}\beta_{i_1} \subset \partial (A_v\times S^1)$ and $N_{i_1}\gamma _{i_2}
- N_{i_2}\beta_{i_2} \subset \partial (A_u\times S^1)$.  In conclusion, the
pages of the individual open book decompositions glue together when performing
the plumbing operation.  This procedure therefore results a horizontal open
book decomposition of $Y=Y_{\Gamma}$ with binding number $\mathbf{n}=(n_v)$.

  Finally we need to check that this open book decomposition is
  compatible with the contact structure $\xi _C$.  More precisely, we
  have to check that the Reeb vector field for a contact form
  representing $\xi_{C}$ is transverse to the pages and tangent to the
  binding components.  Note that Reeb vector field can be chosen to be
  a positive multiple of $\partial_{\theta}$ on $f_v^{-1}(t)$
  (cf. \cite[Proposition~4.2]{G-S}), and a positive multiple of $b_1
  \partial_{q_1} + b_2\partial_{q_2}$ for some $b_1,b_2>0$ on
  $f_e^{-1}(t)$. (Here we follow the notations in \cite[Proposition~4.2]{G-S}
  for $f_v^{-1}(t)$, $f_e^{-1}(t)$, $q_1, q_2$. In the above proof $ \{q_1,
  q_2\}=\{ \gamma _j, \beta_j \}$ on the neighborhood $D_j$.) Therefore the
  open book decomposition is horizontal, concluding the proof.
\end{proof}
As a corollary of the arguments given above, now we can show that
the two contact structures $\xi _M$ and $\xi _C$ are contactomorphic.

\begin{corollary}\label{cor:xiM=xiC}
Suppose that $C\subset (X, \omega )$ is a configuration of symplectic
2-manifolds as before, with $\omega$-convex neighbourhood $U_C$ and
induced contact structure $\xi _C$ on $\partial U_C$.  Let $\xi _M$ be
the Milnor fillable contact structure on the link of a singularity
with resolution graph $\Gamma _C$.  Then $\xi _M$ and $\xi _C$
are contactomorphic.
\end{corollary}
\begin{proof}
Let $\{ E_v \}$ denote the irreducible components of the exceptional
curve in the resolution. These curves correspond to the vertices of
the resolution graph $\Gamma _C$.  Let $D=\sum d_i E_i$ be an
effective divisor satisfying the assumptions of by
Proposition~\ref{pro-CNP}.  (By \cite[Remark~4.1]{C-N-P} such $D$
always exists.) As it is verified by Proposition~\ref{pro-CNP}, the
existence of $D$ shows that there is a horizontal open book
decomposition on $Y_{\Gamma _C}$ compatible with $\xi _M$ which has
$n_v=-D\cdot E_v>0 $ binding component at the vertex $v$.

Define the vector ${\bf {N}}=(N_v)$ of positive rational numbers by the
identity $\mathbf{N} \cdot I= -\mathbf{n}$, where $I$ is the intersection
matrix of plumbing graph $\Gamma _C$ and $\mathbf{N}=(N_v)$,
$\mathbf{n}=(n_v)$.  Suppose that with the choice $k\in {\mathbb {N}}$ the
products $k\cdot N_v$ are integers for all $v$. Consider the horizontal open
book decomposition corresponding to the divisor $k\cdot D$. (As
Lemma~\ref{lemma-multiple} shows, this divisor also satisfies the assumptions
of Proposition~\ref{pro-CNP}.) This procedure provides a horizontal open book
decomposition of $Y_{\Gamma _C}$ which is compatible with $\xi _M$ and has
$kn_v>0$ binding components at each vertex $v$ of $\Gamma _C$.

Now apply Theorem~\ref{theorem-HOB} with the choice $ {\mathbf {n}}$ and $k$
as above. As a result, we get a horizontal open book decompositions compatible
with $\xi _C$ having $kn_v$ binding components at each vertex $v$. Therefore
the two contact structures $\xi _M$ and $\xi _C$ are compatible with
horizontal open book decompositions with equal (and positive) binings, hence
by Theorem~\ref{thm:equiv} the structures are contactomorphic, concluding the
proof.
\end{proof}

With these results at hand, now we can turn to the proof of the main result of
the paper:

\begin{proof}[Proof of Theorem~\ref{thm:main}]
Let $C=(C_1, \ldots , C_m)$ be the given set of symplectic surfaces in
$(X, \omega )$, and $W_C$ a smoothing of a singularity with resolution
graph $\Gamma _C$ given by the configuration $C$. Let $U_C$ be an
$\omega$-convex neighbourhood of $C$ in $X$ (the existence of which is
proved in \cite[Theorem~1.2]{G-S}). According to
Corollary~\ref{cor:xiM=xiC} the contact structure $\xi _C$ induced on
$\partial U_C$ is contactomorphic to the Milnor fillable contact
structure $\xi _M$ on $\partial W_C$, hence by the symplectic gluing
theorem described in \cite{Etn} (see also \cite[Theorem~7.1.9]{OS}),
the manifold $X_C=(X-{\rm {int }} U_C)\cup _{\phi} W_C$ admits a
symplectic structure $\omega _C$ which on $X- U_C$ conicides with the
given symplectic structure $\omega$. The proof is complete.
\end{proof}

\section{An example}
\label{sec:example}

In this section we show an example of a family of singularities with
resolution involving high genus curves, and for which the topological
data of smoothings can be computed.  We will perform the symplectic
surgery on symplectic 4-manifolds using these singularities and the
smoothings.

\subsection*{The singularity}
Let $s$, $t$ and $N$ be positive integers such that $N-1$ divides
$s+t$ and $\gcd(N-1,t)=1$. Consider the hypersurface singularity
$(S,0)=(S_{s,t,N},0)$ given by the equation
\begin{equation}\label{eq:sing}
(x^s+y^s)(x^t+y^{Nt})+z^{N-1}=0.
\end{equation}
Repeatedly blow up the singular curve $(x^s+y^s)(x^t+y^{Nt})=0$ and then
considering the ramified $(N-1)$-fold cover. After normalization and
desingularization, and finally blowing down the rational
$(-1)$-curves, we get the minimal resolution of the singularity $(S,0)$ with
the following properties (cf. \cite{Nem, Nem2}): The resolution consists of
the union of two curves $A$ and $B$, intersecting each other transversally
once, $A^2= -N$ and $g(A)= (s-1)(N-2)/2$ while $B^2=-1$ and
$g(B)=(t-1)(N-2)/2$.

Specializing to $s=3$ and $t=30N-33$, the curve  $A$ is of genus $N-2$
and B is of genus $15N^2-47N+34$.  The condition $\gcd(N-1,t)=1$ is
satisfied when $N-1$ is not divisible by $3$.

\subsection*{Topological data of the smoothing}
We start with a short generic discussion about the computation of
topological data of the Milnor fiber of a hypersurface singularity.
Suppose therefore that $f:(\mathbb{C}^3,0) \to (\mathbb{C},0)$ defines
the isolated singularity $(S,0)$ and $ p:({\tilde {S}},E) \to
(S,0)=(f^{-1}(0),0)$ is its minimal good resolution. We write the
exceptional divisor $p^{-1}(0)=E$ as the union of irreducible
components: $E=E_1\cup \cdots \cup E_m$. Let $h={\rm {rank }} H_1(E)$ and
$p_g=\dim_{\mathbb{C}}H^1({\tilde {S}} ,\mathcal{O}_{{\tilde {S}}})$.
The canonical class $K$ of ${\tilde {S}}$ can be written as $\sum
r_iE_i$, where the $r_i$ are rational numbers, determined by
adjunction formula $2g(E_i)-2=E_i^2 + K \cdot E_i$.  The Milnor number
and the signature of the Milnor fiber of the singularity of
$f^{-1}(0)$ can be computed as follows:
\begin{proposition}[\cite{Dur}]
  The Milnor number $\mu=dim_{\mathbb{C}}\mathbb{C}\{x, y, z\}/(\frac{\partial
    f}{\partial x} ,\frac{\partial f}{\partial y}, \frac{\partial f}{\partial
    z})$ is equal to $\mu=K^2-h+m+12p_g$. The signature of the Milnor fiber
  is equal to $\sigma=-\frac{1}{3}(2\mu+K^2+m+2h)$. \qed
 \end{proposition}

The singularity given by Equation~\eqref{eq:sing} is given as a
ramified cover along a singular plane curve. For an isolated plane
curve singularity the Milnor number satisfies the following equation.

 \begin{proposition}[\cite{Her}]
 For an isolated plane curve singularity $(C,0)\subset (\mathbb{C}^2,0)$
 \begin{equation*}
 \mu(C,0)=d(d-1)+\sum_{x\in Sing(\tilde{C})}\mu(\widetilde{C},x) +1-r,
 \end{equation*}
 where $r$ is the number of different tangent lines of $(C,0)$, $d$
 is the multiplicity of $C$ at $0$ and $\widetilde{C}$ is the proper
 transform of $C$ after one  blow-up at singular point $0$. \qed
 \end{proposition}

Regarding the first Betti number of an isolated singualrity we have
\begin{proposition}[\cite{Gr-St}]\label{prop:grst}
  Let $X_t$ be the Milnor fiber of a smoothing of a pure-dimensional isolated
  normal surface singularity $(X_0,0)$, then $b_1(X_t)=0$. \qed
\end{proposition}

Using the above formulae, for the singularity $(S,0)$ specified by the
function of Equation~\eqref{eq:sing} we have
\begin{itemize}
\item The Milnor number $\mu=b_2(X_t)=(N-2)( (s+t)(s+t-1)+(N-1)t(t-1)+1-s-t)$.
\item
The signature of the Milnor fiber is equal to
$\sigma=-\frac{1}{3}(2\mu +K^2 +m+2h)$.
\end{itemize}
For the specialization $s=3, t=30N-33$ (and $N-1$ is not divisible by 3) these
data become
\begin{itemize}
\item $\mu (f+z^{N-1})=900N^4-3810N^3 +5292N^2-2705N+322$,
\item $\sigma=-300N^4+960N^3- \frac{2348}{3}N^2 + \frac{379}{3}N -2$.
\end{itemize}

\subsection*{Symplectic surgery and the singularity}
Next, we construct a symplectic manifold which contains the curve
configuration $(A,B)$ described above.
According to \cite[Theorem~1]{B-K-M} there is a surface bundle $X\to \Sigma
_{N-2}$ with fiber genus $15N^2-47N+34$ over the surface with genus $N-2$ such
that there is a section with self-intersection $-N$.  Let $M$ denote the
blow-up of $X$ in a fiber. The fiber passing throgh the blown-up point,
together with a section now provides the configuration of two curves $(A, B)$
with intersection patterns as in the resolution graph of the singularity given
by Equation~\eqref{eq:sing}.

Applying the symplectic surgery operation of replacing the
neighbourhood $\nu (A\cup B)$ with the smoothing $W$ of the corresponding
singularity, we get a symplectic 4-manifold $M_W$. Since
the embedding map $A\cup B \to X$ is onto on the first homology,
Proposition~\ref{prop:grst} implies that $b_1(M_W)=0$.

\begin{remark}
Symplectic 4-manifolds containing similar configurations of
symplectic submanifolds can be found near the Bogomolov-Miyaoka-Yau (BMY)
line $c_1^2=9\chi _h$. (For 4-manifold on the BMY line, see \cite{St}.)
We hope that using the symplectic surgery operation
discussed in this paper, one will be able to construct symplectic manifolds
with $b_1=0$ (or even with $\pi _1 =0$) near the BMY line.
We hope to return to this question  in a future project.
\end{remark}

\end{document}